\documentclass{my_elsarticle}

\usepackage{amssymb,amsfonts,amsmath,amsthm}
\usepackage[english]{babel}

\newtheorem{theorem}{Theorem}
\newtheorem{corollary}[theorem]{Corollary}
\newtheorem{definition}[theorem]{Definition}
\newtheorem{lemma}[theorem]{Lemma}

\theoremstyle{remark}

\newtheorem{example}[theorem]{Example}

\begin{document}

\begin{frontmatter}

\author{Amar Chidouh}
\ead{m2ma.chidouh@gmail.com}
\address{Laboratory of Dynamic Systems,
Houari Boumedienne University,
Algiers, Algeria}

\author{Delfim F. M. Torres}
\ead{delfim@ua.pt}
\address{Center for Research and Development in Mathematics and Applications (CIDMA),
Department of Mathematics, University of Aveiro, 3810-193 Aveiro, Portugal}

% ---------------------------------------------

\title{A generalized Lyapunov's inequality for
a fractional boundary value problem\footnote{Part of
first author's Ph.D., which is carried out at
Houari Boumedienne University. See published version at
{\tt http://dx.doi.org/10.1016/j.cam.2016.03.035}}}

% ---------------------------------------------

\begin{abstract}
We prove existence of positive solutions
to a nonlinear fractional boundary value problem.
Then, under some mild assumptions on the nonlinear term,
we obtain a smart generalization of Lyapunov's inequality.
The new results are illustrated through examples.
\end{abstract}

% ---------------------------------------------

\begin{keyword}
fractional differential equations \sep
Lyapunov's inequality \sep
boundary value problem \sep
positive solutions \sep
Guo--Krasnoselskii fixed point theorem.

\medskip

\MSC[2010] 26A33 \sep 34A08.
\end{keyword}
\end{frontmatter}

% ---------------------------------------------

\section{Introduction}

Lyapunov's inequality is an outstanding result in mathematics with many
different applications -- see \cite{MR3371249,MR3352688} and references
therein. The result, as proved by Lyapunov in 1907 \cite{MR1508297}, asserts
that if $q:[a,b]\rightarrow \mathbb{R}$ is a continuous function, then a
necessary condition for the boundary value problem
\begin{equation}  
\label{11}
\begin{cases}
y^{\prime \prime }+qy=0,\ a<t<b,\\
y(a)=y(b)=0
\end{cases}
\end{equation}
to have a nontrivial solution is given by
\begin{equation}  
\label{22}
\int\limits_{a}^{b}\left\vert q(s)\right\vert ds>\frac{4}{b-a}.
\end{equation}
Lyapunov's inequality \eqref{22} has taken many forms, including versions in
the context of fractional (noninteger order) calculus, where the
second-order derivative in \eqref{11} is substituted by a fractional
operator of order $\alpha$.

\begin{theorem}[See \protect\cite{MR3124347}]
\label{thm:rf} 
Consider the fractional boundary value problem
\begin{equation}  
\label{p11}
\begin{cases}
_{a}D^{\alpha }y+qy=0,\ a<t<b,\\
y(a)=y(b)=0,
\end{cases}
\end{equation}
where $_{a}D^{\alpha }$ is the (left) Riemann--Liouville derivative of order
$\alpha \in (1,2]$ and $q:[a,b]\rightarrow \mathbb{R}$ is a continuous
function. If \eqref{p11} has a nontrivial solution, then
\begin{equation}  
\label{i1}
\int\limits_{a}^{b}\left\vert q(s)\right\vert ds
> \Gamma(\alpha )\left( \frac{4}{b-a}\right)^{\alpha-1}.
\end{equation}
\end{theorem}

A Lyapunov fractional inequality \eqref{i1} can also be obtained by
considering the fractional derivative in \eqref{p11} in the sense of Caputo
instead of Riemann--Liouville \cite{ferreira2014lyapunov}. More recently,
Rong and Bai obtained a Lyapunov-type inequality for a fractional
differential equation but with fractional boundary conditions \cite{MR3318923}.  
Motivated by \cite{MR3352131,MR3338866,MR3337865,MR3377951} 
and the above results, as well as existence results on positive solutions
\cite{[01],MR2168413,existence,sub}, which are often useful in applications, 
we focus here on the following boundary value problem:
\begin{equation}  
\label{p}
\begin{cases}
_{a}D^{\alpha }y+q(t)f(y)=0,\ a<t<b, \\
y(a)=y(b)=0,
\end{cases}
\end{equation}
where $_{a}D^{\alpha }$ is the Riemann--Liouville derivative and $1<\alpha \leq 2$.
Our first result asserts existence of nontrivial positive solutions to
problem \eqref{p} (see Theorem~\ref{th}). Then, under some assumptions on
the nonlinear term $f$, we get a generalization of inequality \eqref{i1}
(see Theorem~\ref{th0}).

The paper is organized as follows. In Section~\ref{sec:2} we recall some
notations, definitions and preliminary facts, which are used throughout the
work. Our results are given in Section~\ref{sec:3}: using the
Guo--Krasnoselskii fixed point theorem, we establish in Section~\ref{sec:3.1}
our existence result; then, in Section~\ref{sec:3.2}, assuming that function
$f : \mathbb{R}_{+}\rightarrow \mathbb{R}_{+}$ is continuous, concave and
nondecreasing, we generalize Lyapunov's inequalities \eqref{22} and \eqref{i1}.

% ---------------------------------------------

\section{Preliminaries}
\label{sec:2}

Let $C[a,b]$ be the Banach space of all continuous real functions defined on
$[a,b]$ with the norm $\left\Vert u\right\Vert =\sup_{t\in \lbrack a,b]}
\left\vert u(t)\right\vert$. By $L[a,b]$ we denote the space of all real 
functions, defined on $[a,b]$, which are Lebesgue integrable with the norm
\begin{equation*}
\left\Vert u\right\Vert_{L} 
=\int\limits_{a}^{b}\left\vert u(s)\right\vert ds.
\end{equation*}

The reader interested in the fractional calculus is referred to \cite{MR1347689}. 
Here we just recall the definition of (left) Riemann--Liouville
fractional derivative.

\begin{definition}
The Riemann--Liouville fractional derivative of order $\alpha >0$ of a
function $u:[a,b]\rightarrow \mathbb{R}$ is given by
\begin{equation*}
_{a}D^{\alpha }u(t)=\frac{1}{\Gamma(n-\alpha )}\frac{d^{n}}{dt^{n}}
\int_{a}^{t}\frac{u(s)}{(t-s)^{\alpha -n+1}}ds,
\end{equation*}
where $n=[\alpha ]+1$ and $\Gamma $ denotes the Gamma function.
\end{definition}

\begin{definition}
Let $X$ be a real Banach space. A nonempty closed convex set $P\subset X$ is
called a cone if it satisfies the following two conditions:
\begin{description}
\item[$(i)$] $x\in P$, $\lambda \geq 0$, implies $\lambda x\in P$;

\item[$(ii)$] $x\in P$, $-x\in P$, implies $x=0$.
\end{description}
\end{definition}

\begin{lemma}[Jensen's inequality \cite{MR924157}]
\label{1'} 
Let $\mu$ be a positive measure and let $\Omega$ be a
measurable set with $\mu(\Omega )=1$. Let $I$ be an interval and suppose
that $u$ is a real function in $L(d\mu)$ with $u(t)\in I$ for all 
$t\in \Omega$. If $f$ is convex on $I$, then
\begin{equation}
\label{jenss}
f\left(\int_{\Omega }u(t)d\mu (t)\right) 
\leq \int_{\Omega }(f\circ u)(t)d\mu (t).  
\end{equation}
If $f$ is concave on $I$, then 
the inequality \eqref{jenss} holds with ``$\leq$'' 
substituted by ``$\geq$''.
\end{lemma}

\begin{lemma}[Guo--Krasnoselskii fixed point theorem \cite{[3]}]
\label{kras} 
Let $X$ be a Banach space and let $K\subset X$ be a cone.
Assume $\Omega_{1}$ and $\Omega_{2}$ are bounded open subsets of $X$ with 
$0\in \Omega_{1}\subset \overline{\Omega }_{1}\subset $ $\Omega _{2}$, and
let $T:K\cap (\overline{\Omega }_{2}\backslash \Omega _{1})\rightarrow K$ 
be a completely continuous operator such that
\begin{description}
\item[$(i)$] $\left\Vert Tu\right\Vert \geq \left\Vert u\right\Vert $ for
any $u\in K\cap \partial \Omega _{1}$ and $\left\Vert Tu\right\Vert \leq
\left\Vert u\right\Vert$ for any $u\in K\cap \partial \Omega_{2}$; or

\item[$(ii)$] $\left\Vert Tu\right\Vert \leq \left\Vert u\right\Vert$ for
any $u\in K\cap \partial \Omega_{1}$ and $\left\Vert Tu\right\Vert \geq
\left\Vert u\right\Vert $ for any $u\in K\cap \partial \Omega_{2}$.
\end{description}
Then, $T$ has a fixed point in $K\cap (\overline{\Omega }_{2}\backslash \Omega_{1})$.
\end{lemma}

% ---------------------------------------------

\section{Main results}
\label{sec:3}

Let us consider the nonlinear fractional boundary value problem \eqref{p}
and give its integral representation involving a Green function, which was
deduced in \cite{MR3124347}.

\begin{lemma}
\label{lemmaRF}
Function $y$ is a solution to the boundary value problem \eqref{p} 
if, and only if, $y$ satisfies the integral equation
\begin{equation*}
y(t)=\int\limits_{a}^{b}G(t,s)q(s)f(y(s))ds,
\end{equation*}
where
\begin{equation}
\label{eq:Gf}
G(t,s)=\frac{1}{\Gamma(\alpha)}
\begin{cases}
\frac{(t-a)^{\alpha -1}(b-s)^{\alpha -1}}{(b-a)^{\alpha -1}}
-(t-s)^{\alpha-1},\quad a\leq s\leq t\leq b, \\
\frac{(t-a)^{\alpha -1}(b-s)^{\alpha -1}}{(b-a)^{\alpha -1}},
\quad a\leq s\leq t\leq b,
\end{cases}
\end{equation}
is the Green function associated to problem \eqref{p}.
\end{lemma}

\begin{proof}
Similar to the one found in \cite{MR3124347}.
\end{proof}

\begin{lemma}
\label{lll}
The Green function $G$ defined by \eqref{eq:Gf} 
satisfies the following properties:
\begin{enumerate}
\item $G(t,s)\geq 0$ for all $a\leq t,s\leq b$;

\item $\max_{t\in \lbrack a,b]}G(t,s)=G(s,s),\ s\in \lbrack a,b]$;

\item $G(s,s)$ has a unique maximum given by
\begin{equation*}
\max_{s\in \lbrack a,b]}G(s,s)
=\frac{(b-a)^{\alpha -1}}{4^{\alpha -1}\Gamma(\alpha)};
\end{equation*}

\item there exists a positive function 
$\varphi \in C(a,b)$ such that
\begin{equation*}
\min_{t\in \left[ \frac{2a+b}{3},\frac{2b-a}{3}\right]}G(t,s)
\geq \varphi(s) G(s,s), \quad a<s<b.  
\end{equation*}
\end{enumerate}
\end{lemma}

\begin{proof}
The first three properties are proved in \cite{MR3124347}. Moreover,
we know that Green's function $G(t,s)$ is decreasing with respect
to $t$ for $s\leq t$ and increasing with respect to $t$ for $t\leq s$
\cite{MR3124347}. To prove the fourth property, we define the
following functions:
\begin{equation*}
g_{1}(t,s)=\frac{1}{\Gamma (\alpha )}\left[ 
\frac{(t-a)^{\alpha-1}(b-s)^{\alpha -1}}{(b-a)^{\alpha -1}}
-(t-s)^{\alpha -1}\right] 
\end{equation*}
and
\begin{equation*}
g_{2}(t,s)=\frac{1}{\Gamma (\alpha )}\left[ 
\frac{(t-a)^{\alpha-1}(b-s)^{\alpha-1}}{(b-a)^{\alpha -1}}\right].
\end{equation*}
Obviously, $G(t,s)>0$ for $t,s\in (a,b)$ and one can seek the minimum in an
interval of the form $\left[a+\frac{b-a}{n},b-\frac{b-a}{n}\right]$, 
where $n\geq 3$ is a natural number. 
For $t\in \left[ \frac{2a+b}{3},\frac{2b-a}{3}\right]$,
\begin{equation*}
\begin{split}
\min_{t\in \left[ 
\frac{2a+b}{3},\frac{2b-a}{3}\right] }G(t,s) 
&=
\begin{cases}
g_{1}\left(\frac{2b-a}{3},s\right) 
& \text{ if } s\in (a,\frac{2a+b}{3}], \\
\min \left\{ g_{1}(\frac{2b-a}{3},s),g_{2}\left(\frac{2a+b}{3},s\right)\right\}
& \text{ if } s\in \lbrack \frac{2a+b}{3},\frac{2b-a}{3}], \\
g_{2}\left(\frac{2a+b}{3},s\right) & \text{ if } s\in \lbrack \frac{2b-a}{3},b),
\end{cases} \\
&=
\begin{cases}
g_{1}\left(\frac{2b-a}{3},s\right) 
& \text{ if } s\in (a,\lambda ], \\
g_{2}\left(\frac{2a+b}{3},s\right) 
& \text{ if } s\in \lbrack \lambda ,b),
\end{cases} \\
&=\frac{1}{\Gamma (\alpha )}
\begin{cases}
\left(\frac{(2b-4a)\left(b-s\right)}{3(b-a)}\right)^{\alpha-1}
-\left(\frac{2b-a}{3}-s\right)^{\alpha -1} 
& \text{ if }  s\in (a,\lambda ], \\
\left( \frac{b-s}{3}\right)^{\alpha -1}
& \text{ if } s\in \lbrack \lambda ,b),
\end{cases}
\end{split}
\end{equation*}
where $\frac{2a+b}{3}<\lambda <\frac{2b-a}{3}$ 
is the unique solution of equation
\begin{equation*}
g_{1}\left(\frac{2b-a}{3},s\right)=g_{2}\left(\frac{2a+b}{3},s\right).
\end{equation*}
Set
\begin{equation*}
\varphi (s)=
\begin{cases}
\frac{\left( \frac{(2b-4a)\left( b-s\right) }{3}\right)^{\alpha-1}
-(b-a)^{\alpha -1}(\frac{2b-a}{3}-s)^{\alpha -1}}{((s-a)(b-s))^{\alpha -1}}
& \text{ if } s\in (a,\lambda ],\\
\left(\frac{(b-a)}{3(s-a)}\right)^{^{\alpha -1}}
& \text{ if } s\in \lbrack \lambda ,b).
\end{cases}
\end{equation*}
The proof is complete.
\end{proof}

Let $X=C[a,b]$ and define the operator $T:X\rightarrow X$ as follows:
\begin{equation}  
\label{T}
Ty(t)=\int\limits_{a}^{b}G(t,s)q(s)f(y(s))ds, \quad y\in X.
\end{equation}
To prove existence of solution to the fractional boundary value problem
\eqref{p} it suffices to prove that the map $T$ has a fixed point in $K$.

% ---------------------------------------------

\subsection{Existence of positive solutions}
\label{sec:3.1}

To prove existence of nontrivial positive solutions to the fractional
boundary value problem \eqref{p} we consider the following hypotheses:
\begin{description}
\item[$(H_1)$] $f(y) \geq \overset{\ast }{\gamma}r_{1}$ 
for $y\in \lbrack 0,r_{1}]$,

\item[$(H_2)$] $f(y)\leq \gamma r_{2}$ for $y\in \lbrack 0,r_{2}]$,
\end{description}
where $f:\mathbb{R}_{+}\rightarrow \mathbb{R}_{+}$ is continuous. 
In what follows we take
\begin{equation*}
\gamma :=\left( \int\limits_{a}^{b}
G(s,s)q(s)ds\right)^{-1}
\text{ and } \overset{\ast }{\gamma }
:=\left( \int\limits_{\frac{2a+b}{3}}^{\frac{2b-a}{3}}G(s,s)
\varphi (s)q(s)ds\right)^{-1}.
\end{equation*}

\begin{theorem}
\label{th} 
Let $q:[a,b]\rightarrow \mathbb{R}_+$ 
be a nontrivial Lebesgue integrable function.
Assume that there exist two positive constants $r_{2}>r_{1}>0$
such that the assumptions $(H_1)$ and $(H_2)$ are satisfied. 
Then the fractional boundary value problem \eqref{p} has at 
least one nontrivial positive solution $y$ belonging to $X$ 
such that $r_{1} \leq \left\Vert y\right\Vert \leq r_{2}$.
\end{theorem}

For the proof of Theorem~\ref{th} we use Lemma~\ref{kras} 
with the cone $K$ given by
\begin{equation*}
K:=\left\{y\in X:\ y(t)\geq 0,\ a\leq t\leq b\right\}.
\end{equation*}

\begin{proof}[Proof of Theorem~\protect\ref{th}]
Using the Ascoli--Arzela theorem, we prove that $T:K\rightarrow K$ 
is a completely continuous operator. Let
\begin{equation*}
\Omega_{i}=\left\{ y\in K:\left\Vert y\right\Vert \leq r_{i}\right\}.
\end{equation*}
From $(H_1)$ and Lemma~\ref{lll}, we have for 
$t\in \left[ \frac{2a+b}{3},\frac{2b-a}{3}\right]$ 
and $y\in K\cap \partial \Omega_{1}$ that
\begin{equation*}
\begin{split}
(Ty)(t) 
&\geq \int\limits_{a}^{b}\min_{t\in \left[ 
\frac{2a+b}{3},\frac{2b-a}{3}\right]} G(t,s)q(s)f(y(s))ds \\
&\geq \overset{\ast }{\gamma }\left( \int\limits_{a}^{b}
G(s,s)\varphi(s)q(s)ds\right) r_{1} \\
&\geq \overset{\ast }{\gamma }\left( 
\int\limits_{\frac{2a+b}{3}}^{\frac{2b-a}{3}}
G(s,s)\varphi (s)q(s)ds\right) r_{1}=\left\Vert y\right\Vert.
\end{split}
\end{equation*}
Thus, $\left\Vert Ty\right\Vert \geq \left\Vert y\right\Vert$ for
$y\in K\cap \partial \Omega _{1}$. Let us now prove that $\left\Vert
Ty\right\Vert \leq \left\Vert y\right\Vert $ for all $y\in K\cap \partial
\Omega _{2}$. From $(H_2)$, it follows that
\begin{equation*}
\left\Vert Ty\right\Vert 
=\max_{t\in \lbrack a,b]}\int\limits_{a}^{b}
G(t,s)q(s)f(y(s))ds\leq \gamma \left(\int
\limits_{a}^{b}G(s,s)q(s)ds\right) r_{2}
=\left\Vert y\right\Vert
\end{equation*}
for $y\in K\cap \partial \Omega _{2}$. Thus, from Lemma~\ref{kras}, we
conclude that the operator $T$ defined by \eqref{T} has a fixed point in 
$K\cap (\overline{\Omega }_{2}\backslash \Omega _{1})$. Therefore, the
fractional boundary problem \eqref{p} has at least one positive solution 
$y$ belonging to $X$ such that $r_{1}\leq \left\Vert y\right\Vert \leq r_{2}$.
\end{proof}

\begin{example}
Consider the following fractional boundary value problem:
\begin{equation}
\label{pex}
\begin{cases}
_{0}D^{3/2}y+te^{y}=0 & \text{ if } 0<t<1,\\
y(0)=y(1)=0.
\end{cases}
\end{equation}
Firstly, let us calculate the values of $\gamma$ 
and $\overset{\ast }{\gamma}$. Here,
\begin{equation*}
\varphi(s)=
\begin{cases}
\frac{\sqrt{\frac{2\left( 1-s\right) }{3}}
-\sqrt{\frac{2}{3}-s}}{\sqrt{s(1-s})}
& \text{ if } s\in (0,\lambda ],\\
\frac{1}{\sqrt{3s}} 
& \text{ if } s\in \lbrack \lambda,1),
\end{cases}
\end{equation*}
where $\lambda \simeq 0.64645$. 
Hence, by a simple computation, we get
\begin{equation}
\label{eq:val:g}
\overset{\ast }{\gamma }\simeq 26.459 
\text{ and } \gamma \simeq 4.\,\allowbreak 514.  
\end{equation}
We choose $r_{1}=\frac{1}{27}$ and $r_{2}=1$. Then we get
\begin{enumerate}
\item $f(y)=e^{y}\geq \overset{\ast }{\gamma }r_{1}$ 
for $y\in \lbrack 0, \frac{1}{27}]$;

\item $f(y)=e^{y}\leq \gamma r_{2}$ for $y\in \lbrack 0,1]$.
\end{enumerate} 
Therefore, from Theorem~\ref{th}, problem \eqref{pex} has 
at least one nontrivial solution $y$ in $X$ such that 
$\frac{1}{27}\leq \left\Vert y\right\Vert \leq 1$.
\end{example}

% ---------------------------------------------

\subsection{Generalized Lyapunov's inequality}
\label{sec:3.2}

The next result generalizes Theorem~\ref{thm:rf}: choosing $f(y)=y$ in
Theorem~\ref{th0}, inequality \eqref{eq:new:ineq} reduces to \eqref{i1}.
Note that $f\in C(\mathbb{R}_{+}, \mathbb{R}_{+})$ is a concave and
nondecreasing function.

\begin{theorem}
\label{th0} 
Let $q:[a,b]\rightarrow \mathbb{R}$ 
be a real nontrivial Lebesgue integrable function.
Assume that $f\in C\left( \mathbb{R}_{+},\mathbb{R}_{+}\right)$ is a
concave and nondecreasing function. If the fractional boundary value 
problem \eqref{p} has a nontrivial solution $y$, then
\begin{equation}
\label{eq:new:ineq}
\int\limits_{a}^{b} |q(t)| dt
>\frac{4^{\alpha -1}\Gamma (\alpha )\eta }{(b-a)^{\alpha -1}f(\eta)},  
\end{equation}
where $\eta =\max_{t\in \lbrack a,b]}y(t)$.
\end{theorem}

\begin{proof}
We begin by using Lemma~\ref{lemmaRF}. We have
\begin{equation*}
\begin{split}
\left\vert y(t)\right\vert 
&\leq \int\limits_{a}^{b}G(t,s) |q(s)| f(y(s))ds,\\
\left\Vert y\right\Vert 
&\leq \int\limits_{a}^{b}G(s,s) |q(s)| f(y(s))ds 
<\frac{(b-a)^{\alpha -1}}{4^{\alpha -1}\Gamma(\alpha)} 
\int\limits_{a}^{b} |q(s)| f(y(s)) ds.
\end{split}
\end{equation*}
Using Jensen's inequality \eqref{jenss}, and taking into account that 
$f$ is concave and nondecreasing, we get that
\begin{equation*}
\left\Vert y\right\Vert 
<\frac{(b-a)^{\alpha-1}\left\Vert q\right\Vert_{L}}{4^{\alpha -1}
\Gamma(\alpha)}\int\limits_{a}^{b}
\frac{|q(s)|f(y(s))ds}{\left\Vert q\right\Vert_{L}} \\
<\frac{(b-a)^{\alpha -1}\left\Vert q\right\Vert _{L}}{4^{\alpha -1}
\Gamma(\alpha)}f(\eta),
\end{equation*}
where $\eta =\max_{t\in \lbrack a,b]}y(t)$. Thus,
\begin{equation*}
\int\limits_{a}^{b} |q(s)| ds
>\frac{4^{\alpha -1} \Gamma(\alpha )\eta }{(b-a)^{\alpha -1}f(\eta)}.
\end{equation*}
This concludes the proof.
\end{proof}

\begin{corollary}
\label{co} 
Consider the fractional boundary value problem
\begin{equation*}
\begin{cases}
_{a}D^{\alpha }y+q(t)f(y)=0,\quad a<t<b,\\
y(a)=y(b)=0,
\end{cases}
\end{equation*}
where $f\in C\left( \mathbb{R}_{+},\mathbb{R}_{+}\right)$ is 
concave and nondecreasing and $q\in L([a,b],\mathbb{R}_{+}^{\ast})$. 
If there exist two positive constants $r_{2}>r_{1}>0$ such that 
$f(y) \geq \overset{\ast }{\gamma} r_{1}$ 
for $y\in \lbrack 0,r_{1}]$ and $f(y)\leq \gamma r_{2}$ 
for $y\in \lbrack 0,r_{2}]$, then
\begin{equation*}
\int\limits_{a}^{b}q(t)dt
>\frac{4^{\alpha -1}\Gamma (\alpha )r_{1}}{(b-a)^{\alpha -1}f(r_{2})}.
\end{equation*}
\end{corollary}

\begin{example}
Consider the following fractional boundary value problem:
\begin{equation*}
\begin{cases}
_{0}D^{3/2}y + t \ln (2+y)=0, \quad 0<t<1, \\
y(0)=y(1)=0.
\end{cases}
\end{equation*}
We have that
\begin{description}
\item[$(i)$] $f(y)=\ln (2+y): \mathbb{R}_{+}\rightarrow \mathbb{R}_{+}$ is
continuous, concave and nondecreasing;

\item[$(ii)$] $q(t)=t:[0,1]\rightarrow \mathbb{R}_{+}$ is a Lebesgue
integral function with $\left\Vert q\right\Vert_{L}=1>0$.
\end{description}
We computed the values of $\gamma$ and $\overset{\ast }{\gamma}$ 
in \eqref{eq:val:g}. Choosing $r_{1}=1/40$ and $r_{2}=1$, we get
\begin{enumerate}
\item $f(y)=\ln (2+y)\geq \overset{\ast }{\gamma }r_{1}$ 
for $y\in \lbrack 0,1/40]$;

\item $f(y)=\ln (2+y)\leq \gamma r_{2}$ for $y\in \lbrack 0,1]$.
\end{enumerate}
Therefore, from Corollary~\ref{co}, we get that
\begin{equation*}
\int\limits_{0}^{1}q(t)dt
>\frac{4^{\alpha -1}\Gamma(\alpha )r_{1}}{(b-a)^{\alpha-1}f(r_{2})}
\simeq 4.0334\times 10^{-2}.
\end{equation*}
\end{example}

% ---------------------------------------------

\section*{Acknowledgments}

This research was initiated while Chidouh was visiting the Department of
Mathematics of University of Aveiro, Portugal, 2015. The hospitality of the
host institution and the financial support of Houari Boumedienne University,
Algeria, are here gratefully acknowledged. Torres was supported through
CIDMA and the Portuguese Foundation for Science and Technology (FCT), within
project UID/MAT/04106/2013. The authors are grateful to Rachid
Bebbouchi for putting them in touch.
They are also very grateful to two anonymous referees 
for their careful reading of the submitted manuscript 
and for their valuable comments and questions.

% ---------------------------------------------

% ---------------------------------------------

\end{document}